\documentclass[12pt]{article}
\usepackage{dsfont, amsfonts,amsthm,tikz}
\usepackage{graphicx, parskip,amsmath}
\usepackage{amsmath, amssymb}
\usepackage{geometry,setspace}
\usepackage{breqn}
\usepackage{rotating,lineno}
\usepackage{multicol,array,multirow,caption}
\usepackage{titlesec}
\usepackage{enumitem}
\usepackage{tabstackengine}
\setstackEOL{\cr}
\setlist{  
  listparindent=\parindent,
  parsep=0pt,
}
\titleformat*{\subsection}{\bfseries}
\font\head=cmbx10 scaled 1728  

\long\def\symbolfootnote[#1]#2{\begingroup%
\def\thefootnote{\fnsymbol{footnote}}\footnote[#1]{#2}\endgroup}
\newtheorem{thm}{Theorem}[section]

\newtheorem{example}{Example}

\usepackage{algorithm2e}

\SetKwInOut{Parameter}{parameter}

\doublespace

\begin{document}
\vskip 18pt {\bf {\head \centerline{Some Partitionings of Complete Designs}
}
\vskip 18pt \centerline{M.H. Ahmadi$^a$,\ N. Akhlaghinia$^b$,
\ G.B. Khosrovshahi$^{a,}$\symbolfootnote[1]{Corresponding author: rezagbk@ipm.ir}, \ S. Sadri$^a$}} \vskip 6pt
{\centerline{{\it $^a$School of Mathematics, Institute for
Research in Fundamental Sciences} (IPM)}
 \vskip 6pt\centerline{{\it P.O. Box
$19395-5746$, Tehran, Iran}} \vskip 2pt

{\centerline{{\it $^b$Department of Mathematical Sciences Shahid, Beheshti University, G.C.,\\
}}
 \vskip 6pt\centerline{{\it P.O. Box
$19839-63113$, Tehran, Iran}} \vskip 2pt

\centerline{\{h.ahmadi117, narges.nia, rezagbk, sadri\}@ipm.ir}
\vskip 2pt
\centerline{Mathematics Subject Classifications: 05B20}
\vskip 5pt

\hrule 
\begin{abstract}
Let $v\geq6$ be an integer 
with $v\equiv2 \pmod 4$.
 In this paper, we introduce a new partitioning of the set of all $3$-subsets of a $v$-set  into some simple trades. 
\end{abstract}


\section{Introduction}
Integers $t, k$, and  $v$ with $0 \leq t \leq k \leq v-t$ are considered. Let
$X=\{0,1,\dots,v-1\}$ be a linearly ordered $v$-set, and
\[ \binom{X}{i} :=\{ A \subseteq X : |A|=i \}, \qquad 0 \leq i \leq v. \]
The
elements of $\binom{X}{k}$  are called {\it blocks}.
For the sake of brevity, sometimes a set $\{a_0,a_1, \dots , a_i\}$ is denoted by the string
$ a_0a_1 \dots a_i $.

The {\it inclusion matrix} $W_{tk} (v)$ (known as Wilson Matrix)  is defined to be a $\binom{v}{t}$ by $\binom{v}{k}$ 
$(0,1)$-matrix whose rows and columns are indexed by (and referred to)
the members of $\binom{X}{t}$ and $\binom{X}{k}$, respectively, 
where
 \begin {equation*} 
  W_{tk}(v)(T,K):=
  \left \{
    \begin{array}{r@{\qquad}l}
      1 & {\rm if} \ T\subseteq K\\
      0 & {\rm otherwise}
    \end{array}
  \right. \quad T \in {\binom{X}{t}},\ K \in {\binom{X}{k}}.
\end {equation*}
For the sake of convenience, sometimes we use $W_{tk}$ or just a bare $W$ for $W_{tk}(v)$.

Now, suppose that the elements of  $\binom{X}{k}$ are
lexicographically ordered. 
A $\mathbb{Z}$\,-{\it collection} of the elements
of $\binom{X}{k}$ is a function $f:\binom{X}{k}\to \mathbb{Z}$  with the vector representation $\big(f(A_1),\dots,f(A_{\binom{v}{k}})\big)^T$.

It is well known that $W_{tk}$ is a full rank matrix over
$\mathbb{Q}$ \cite{GraverJurkat}. As a linear
operator, $W_{tk}$ acts on a $\mathbb{Z}$\,-collection of blocks, and algebraically counts the number of
times that any member
of $\binom{X}{t}$ appears in the blocks of the collection.

Let {\bf 1} be the all 1 vector, and let $\lambda$ be a nonnegative integer. We call the following equation  the {\it fundamental equation} of design theory:
\begin{equation}\label{eq}
W_{tk}\, f=\lambda {\bf 1}.\tag{*}
\end{equation}

For $\lambda>0$, 
every nonnegative integral solution of Equation (\ref{eq}) is called a $t$-$(v,k,\lambda)$ {\it  design}.
Also, a simple $2$-$(v,3,1)$ design is called a {\it Steiner triple system}, denoted by ${\rm STS}(v)$. It is well known that an ${\rm STS}(v)$ exists if and only if 
$v\equiv 1,3\pmod 6$.
For more on Steiner triple systems, see \cite{TripleSystems}.

For $\lambda=0$, 
every integral solution of Equation (\ref{eq}) is called a
${\rm T}(t,k,v)$ {\it trade}. Let 
$T$ 
be a ${\rm T}(t,k,v)$ trade. 
Clearly $\rm T$ has some negative and positive entries. Therefore, we can assume that $\rm T = T_0-T_1$,
where $\rm T_0$ and $\rm T_1$ are two nonnegative $\mathbb{Z}$\,-collections.
Now, $W_{tk}\ {\rm{T}}_0=W_{tk}\ {\rm{T}}_1$ certifies that every element of $\binom{X}{t}$ appears equally often in $\rm T_0$ and $\rm T_1$. This means that $\rm T_0$ and $\rm T_1$, which are called the two {\it legs} of T, are mutually $t$-wise balanced.
The number of blocks in $T_0$ and $T_1$ are equal and is called the {\it volume} of $T$.

The {\it foundation} of a trade is the set of $s$ elements from $X$ appearing in the trade.
A nonzero
T$(t,k,v)$ trade  whose both volume and foundation size take the minimum
values is called a {\it minimal trade}. It is known that the volume
and foundation size of a T$(t,k,v)$ trade are at least  $2^t$ and $k+t+1$, respectively \cite{HandbookKhosro}. 
A trade is called {\it simple} if  it contains no repeated blocks.

In what follows, we present two kinds of simple trades.
\begin{itemize}
\item[$\bullet$]
Let 
$a_0a_1a_2$ 
and 
  $b_0b_1b_2$ 
  be two disjoint blocks. Then 
\begin{equation*}\label{minimal}
 {a_0\ a_1\ a_2 \choose b_0\ b_1\ b_2}:=(a_0-b_0)(a_1-b_1)(a_2-b_2),
\end{equation*}
is a minimal ${\rm T}(2,3,v)$ trade.

\item[$\bullet$] Suppose that 
$a_0a_1\dots a_{n-1}$
and 
$b_0b_1\dots b_{n-1}$، 
are two disjoint subsets of 
$X$, and $3\leq n \leq \frac{v}{2}$. Let $G$ be a graph with vertices $\{a_0b_0,a_1b_1,\dots,a_{n-1}b_{n-1}\}$.  For every cycle $\mathcal{C}$ of $G$ 
\begin{equation*}\label{VequalFTrade}
T(\mathcal{C}):=\sum_{(a_ib_i,a_jb_j)\in \mathcal{C}}a_ib_i(a_j+b_j)-a_jb_j(a_i+b_i),
 \end{equation*}
is a ${\rm T}(2,3,v)$ trade.

\end{itemize}

In Section \ref{PartitionMethod}, we construct a simple ${\rm T}(2,3,v)$ trade with volume 
$\binom{v}{3}/2$ by using disjoint trades. For example, Let $X=\{a_0,a_1,a_2,b_0,b_1,b_2\}$, and $(a_0b_0,\, a_1b_1,\,a_2b_2,\,a_0b_0)$ be a cycle of graph $G$ with vertices $\{a_0b_0,a_1b_1,a_2b_2\}$, then 
\[
\begin{pmatrix}
a_0&a_1&a_2\\
b_0&b_1&b_2
\end{pmatrix}
{\text{\Large $+$}}\;
T(a_0b_0,\, a_1b_1,\,a_2b_2,\,a_0b_0)
\] 
is a simple ${\rm T}(2,3,6)$ trade with volume $10$.

A simple $\rm{T}(t,k,v)$ trade with volume  $\binom{v}{k}/2$ is called a $(t,k,v)$-{\it halving}.  
The following conjecture (known as {\it halving conjecture}) is due to Alan Hartman \cite{Hartman}:
\begin{quotation}
\noindent
{\bf Conjecture.}
\it There exists a  $(t,k,v)$-halving if and only if 
for $0\leq i\leq t$,
$\binom{v-i}{k-i}\equiv 0\pmod2$.
\end{quotation} 

Clearly, 
every $(t,k,v)$-halving can be written as a linear combination of some trades with smaller volumes. 
  In what follows, 
we  briefly describe three  methods available in the literature  for constructing a $(2,3,v)$-halving. 
All methods employ the linear combination of trades.  

\begin{itemize}
\item[$\circ$] {\bf AK algorithm \cite{NewBasis}.}
For a block $a_0a_1a_2$, we  choose the block  $b_0b_1b_2$ such that 
\[\begin{split}
    b_2&=a_2+1,\\
    b_1&=\min\Big(\big\{x\in X\setminus\{a_0,a_1,a_2,b_2\}\, \mid x>a_1\big\}\Big),\\
    b_0&=\min\Big(\big\{x\in X\setminus \{a_0,a_1,a_2,b_2,b_1\}\,\mid x>a_0\big\}\Big).
    \end{split}\]
Then we define the trade 
\[\mathfrak{T}_{a_0a_1a_2}:={a_0\ a_1\ a_2 \choose b_0\ b_1\ b_2}.\]

Ajoodani and Khosrovshahi 
proved that the following algorithm produces a $(2,3,v)$-halving.
\begin{algorithm}
\Begin{
$T :=\mathfrak{T}_{012}$
\While{T is not a $(2,3,v)$-halving}{
  Find first block $B$ on lexicographical order such that $B\not\in T$
  \lIf{$T+\mathfrak{T}_B$ is simple}{$T:= T+\mathfrak{T}_B$}
\lElse{$T:= T-\mathfrak{T}_B$}
}
\KwRet $T$
}
 \end{algorithm}
 
 Although, through this algorithm the halving is constructed directly, but it does not reveal the trade-like structure of the halving. 
\item[$\circ$] {\bf Standard recursive method \cite{MoreOnWilson}.} 
By reordering the columns of $W$, 
we can write the 
 reduced row echelon form of  $W$  as 
$\big(C\mid {\rm I}\big)$. 

Then  
$\big( \frac{\rm ~I}{-C} \big)$
is called {\it the standard basis} for the kernel of $W$ and denoted by 
$\mathbb{S}_{tk}(v)$.

Naturally,  
every $(t,k,v)$-halving is a linear combination of the columns of $\mathbb{S}_{tk}(v)$. 
The complicated structure of $\mathbb{S}_{tk}(v)$ does not reveal much information about the structure of $(t,k,v)$-halving and in this regard 
 the following two conjectures remain open  \cite{MoreOnWilson, Wong}.
\begin{itemize}
\item The elements of every row of $\mathbb{S}_{tk}(v)$ have the same sign.
\item For $t>1$, the matrix $\mathbb{S}_{tk}(v)$ contains a nowhere zero row.
\end{itemize}

Nevertheless, 
in
\cite{Ahmadi}, 
by carefully studying $\mathbb{S}_{23}(v)$, 
a $(1,-1)$-vector $\eta$ is constructed recursively  such that  
$\mathbb{S}_{23}(v)\eta=h$, and $h$ is a 
$(2,3,v)$-halving.

In this method $\binom{v}{3}-\binom{v}{2}$
trades are used which are not  necessarily simple and every two trade are not disjoint.
\item[$\circ$] {\bf V10 Method \cite{Wong}.}
Let  
 $v=4n+2$ 
  and 
$X$ 
 is partitioned into two subsets 
 $a_0a_1\dots a_{2n}$ 
and 
 $b_0b_1\dots b_{2n}$.
 If $\alpha_0\alpha_1\alpha_2$ is a subset of 
 $\{0,1,\dots ,2n\}$, then we define the trade 
\[\mathcal{T}_{\alpha_0\alpha_1\alpha_2}:=\begin{pmatrix}
a_{\alpha_0}& a_{\alpha_1} &a_{\alpha_2}\\
b_{\alpha_0}& b_{\alpha_1} &b_{\alpha_2}
\end{pmatrix}
+ 
\begin{pmatrix}
a_{\alpha_0}& a_{\alpha_1} &a_{\alpha_2}\\
b_{\alpha_1}& b_{\alpha_2} &b_{\alpha_0}
\end{pmatrix}
-
\begin{pmatrix}
a_{\alpha_0}& a_{\alpha_1} &a_{\alpha_2}\\
b_{\alpha_2}& b_{\alpha_0} &b_{\alpha_1}
\end{pmatrix}.
\]

It is easy to show that
$\mathcal{T}_{\alpha_0\alpha_1\alpha_2}$   
is a simple trade with volume $10$ and foundation size $6$. 
The following summation gives  a $(2,3,v)$-halving
\[
\sum_{\substack{\alpha_0\alpha_1\alpha_2\subseteq \{0,\dots,2n\}}}
\mathcal{T}_{\alpha_0\alpha_1\alpha_2} 
 \times (-1)^{\alpha_0+\alpha_1+\alpha_2}.
\]
 
\noindent{\bf Example.} The following summation is a $(2,3,10)$-halving
\[\mathcal{T}_{123}-\mathcal{T}_{124}+\mathcal{T}_{125}+\mathcal{T}_{134}-\mathcal{T}_{135}+\mathcal{T}_{145}-\mathcal{T}_{234}+\mathcal{T}_{235}-\mathcal{T}_{245}+\mathcal{T}_{345}.\]
For instance
\begin{dmath*}
\setstacktabbedgap{2.7pt}
\mathcal{T}_{235}=
\parenMatrixstack{
2&3&5\cr
7&8&10
}
+\parenMatrixstack{
2&3&5\cr
8&10&7
}
-\parenMatrixstack{
2&3&5\cr
10&7&8
}\cdot
\end{dmath*}

In this  method,  the linear combination of  
 $\binom{\frac{v}{2}}{3}$ 
trades with volume $10$ and foundation size $6$ is used to construct 
a $(2,3,v)$-halving. 
Every two trades are disjoint or  have exactly $4$ blocks in common.
\end{itemize}
\section{%
Partition method
}\label{PartitionMethod} 

Recently,  Sauskan and Tarannikov
partitioned $\binom{\{1,\dots,10\}}{3}$  into 
$15$ disjoint minimal trades and subsequently by augmenting these trades they obtained a $(2,3,10)$-halving  \cite{Packing}.  
 Then, for $v=4(2n)+2$, they constructed a $(2,3,v)$-halving by utilizing these trades and the method of {\it combining $t$-designs}.
 Originally, this method was used to obtain some infinite families of $t$-designs \cite{Combining}.

In \cite{Packing}, 
 for constructing  $(2,3,10)$-halving, no algorithm  has  been presented. Here,
we construct $(2,3,10)$-halvings based on a {\it hill climbing} process. 

\begin{algorithm}[H]
\Begin{
$M:=\binom{\{1,\dots,10\}}{3}$\\
$n:=0$\\
$H:=\varnothing$\\
\While{$n\neq 15$}{
  \eIf{there is a $B$ in $M$ such that $T_B\subseteq M$}{
 $H:=H\cup\{T_B\}$\\
  $n:=n+1$\\
  $M:=M\setminus T_B$}
{choose  trade $T$ from $H$\\
$H:=H\setminus\{T\}$\\
  $n:=n-1$\\
  $M:=M\cup T$}
}
}
\end{algorithm}
\begin{example}
 A ${(2,3,10)}$-halving.
\begin{footnotesize}
\begin{dmath*}
\setstacktabbedgap{2.7pt}
\parenMatrixstack{
3&4&6\cr
8&7&9
}
+\parenMatrixstack{
2&5&10\cr
6&3&9
}
+\parenMatrixstack{
1&3&5\cr
2&4&6
}
+\parenMatrixstack{
2&3&8\cr
5&9&4
}
+\parenMatrixstack{
3&7&10\cr
6&1&2
}
+\parenMatrixstack{
1&2&9\cr
10&3&4
}
+\parenMatrixstack{
4&7&8\cr
1&5&6
}
+\parenMatrixstack{
1&2&7\cr
4&3&8
}\\
+\parenMatrixstack{
4&7&9\cr
2&6&10
}
+\parenMatrixstack{
3&5&10\cr
1&8&9
}
+\parenMatrixstack{
3&6&8\cr
2&7&5
}
+\parenMatrixstack{
1&7&9\cr
5&10&4
}
+\parenMatrixstack{
8&9&10\cr
7&5&6
}
+\parenMatrixstack{
1&7&10\cr
8&2&5
}
+\parenMatrixstack{
6&8&10\cr
4&9&1
}\cdot
\end{dmath*}
\end{footnotesize}
\end{example}

We note that a $(2,3,v)$-halving exists if and only if $v=4n+2$. 
Clearly, by augmetation of minmal trades
for $v=4(2n-1)+2$,
one can not  construct a $(2,3,v)$-halving.
In what follows, we describe our partitioning mehtod to construct $(2,3,v)$-halvings.

\begin{thm}\label{DisjointTrades}
For positive integer $n$,
let 
$v=4n+2$ 
 and suppose that 
$X$ is  partitioned into two subsets  
 $a_0a_1\dots a_{2n}$ 
and 
 $b_0b_1\dots b_{2n}$.   Let $\mathcal{C}$ be the Eulerian cycle of  complete graph $G$ with vertices $\{a_0b_0,a_1b_1,\dots,a_{2n}b_{2n}\}$, then 
\[
h=\sum_{\substack{\alpha_0\alpha_1\alpha_2\subseteq \{0,\dots,2n\}}}
\begin{pmatrix}
a_{\alpha_0}& a_{\alpha_1} &a_{\alpha_2}\\
b_{\alpha_0}& b_{\alpha_1} &b_{\alpha_2}
\end{pmatrix} 
+
\;T(\mathcal{C})
\] 
is a 
$(2,3,4n+2)$-halving. 
\end{thm}

\begin{proof}
It is easy to check that all the blocks of $h$ are disjoint. Then by the following relation 
\[\binom{2n+1}{3}\times 4+\binom{2n+1}{2}\times2=\frac{\binom{4n+2}{3}}{2},\]
clearly 
$h$
is 
a $(2,3,4n+2)$-halving.
\end{proof}

In the following Theorem, we partition all blocks of $T(\mathcal{C})$ of Theorem \ref{DisjointTrades} into trades with volume $6$ and volume $8$.
\begin{thm}
For positive integer $n$,
let 
$v=4n+2$ 
 and suppose that 
$X$ is  partitioned into two subsets  
 $a_0a_1\dots a_{2n}$ 
and 
 $b_0b_1\dots b_{2n}$.  Then 
\begin{itemize}
\item[$(i)$]
Let
$\,2n+1{\equiv} 1,3\pmod6$. 
Consider an ${\rm STS}(2n+1)$ with elements from $\{a_0,a_1,\dots,a_{2n}\}$. 
Then
\[
h_1=\sum_{\substack{\beta_0\beta_1\beta_2\in {\rm STS}(2n+1)}}
T(a_{\beta_0}b_{\beta_0},\,a_{\beta_1}b_{\beta_1},\,a_{\beta_2}b_{\beta_2},\,a_{\beta_0}b_{\beta_0})
\] 
is a 
${\rm T}(2,3,4n+2)$ trade with volume $2\binom{2n+1}{2}$. 

\item[$(ii)$]
Let
$\,2n+1{\equiv} 5\pmod6$.
Consider 
${\rm STS}(2n-1)$ 
with elements from  
$\{a_0,\dots,a_{2n-2}\}$.
If $S$ is 
a partition of
$\{a_0,\dots,a_{2n-3}\}$ 
into $2$-subsets, then 
\begin{align*}
h_2=&\sum_{\substack{\beta_0\beta_1\beta_2\in {\rm STS}(2n-1)}}
T(a_{\beta_0}b_{\beta_0},\,a_{\beta_1}b_{\beta_1},\,a_{\beta_2}b_{\beta_2},\,a_{\beta_0}b_{\beta_0})\\
&\;+ 
T(a_{2n-2}b_{2n-2},\,a_{2n-1}b_{2n-1},\,a_{2n}b_{2n},\,a_{2n-2}b_{2n-2})\\
&\;+
\sum_{\substack{a_ia_j\in S}}
T(a_{2n-1}b_{2n-1},\,a_{i}b_{i},\,a_{2n}b_{2n},\,a_{j} b_{j},\,a_{2n-1}b_{2n-1})
\end{align*}
is 
a ${\rm T}(2,3,4n+2)$ trade with volume $2\binom{2n+1}{2}$.
\end{itemize}
\end{thm}

\begin{proof}
\begin{itemize}
\item[$(i)$]
Since the number of blocks of 
${\rm STS}(2n+1)$ 
is 
$\frac{1}{3}\binom{2n+1}{2}$,
then 
 $h_1$ 
is the augmentation of 
$\frac{1}{3}\binom{2n+1}{2}$ trades with volume $6$. 
Now,  by 
the  relation 
\[\frac{1}{3}\binom{2n+1}{2}\times6=2\binom{2n+1}{2},\] 
$h_1$ 
is a 
${\rm T}(2,3,4n+2)$ trade.

\item[$(ii)$]
Since the number of blocks of 
${\rm STS}(2n-1)$ 
is 
$\frac{1}{3}\binom{2n-1}{2}$, 
then 
$h_2$ 
is the augmentation  of  
$\frac{1}{3}\binom{2n-1}{2}+1$
trades of volume 
$6$ 
and 
$n-1$ 
trades with of
 $8$.
Now,  by 
the following relation 
\[\left(\frac{1}{3}\binom{2n-1}{2}+1\right)\times6+(n-1)\times8=2\binom{2n+1}{2},\]
$h_2$ 
is a 
${\rm T}(2,3,4n+2)$ trade.

\end{itemize} 
\end{proof}
\section{%
Some Examples
} 
{\bf 
1. $\mathbf{(2,3,14)}$-halving.
} 
Let 
$a_i=i$ 
and 
$b_i=i+7$, where 
$0\leq i\leq 6$. 
Consider
${\rm STS}(7)=\{013,026,045,124,156,235,346\}.$ 
Then 
\[\begin{aligned}
h=
\sum_{\substack{\alpha_0\alpha_1\alpha_2\subseteq \{0,\dots,6\}}}
\begin{pmatrix}
a_{\alpha_0}& a_{\alpha_1} &a_{\alpha_2}\\
b_{\alpha_0}& b_{\alpha_1} &b_{\alpha_2}
\end{pmatrix} 
\;+
\sum_{\substack{\beta_0\beta_1\beta_2\in {\rm STS}(7)}}
T(a_{\beta_0}b_{\beta_0},\,a_{\beta_1}b_{\beta_1},\,a_{\beta_2}b_{\beta_2},\,a_{\beta_0}b_{\beta_0})\end{aligned}\] 
is a $(2,3,14)$-halving.  
Therefore, the trade structure of $(2,3,22)$-halving is the following:
\[\begin{aligned}35\, ({\rm minimal\; trades})+7\, ({\rm trades \;of\; volume\; 6})=182.\end{aligned}\] 
\\[10pt]
{\bf 
2. $\mathbf{(2,3,22)}$-halving.
}
Let 
$a_i=i$ 
and 
$b_i=i+11$, where 
$0\leq i\leq 10$. 
Consider
${\rm STS}(9)=\{012,036,048,057,138,147,156,237,246,258,345,678\}$ and 
$S=\{01,23,45,67\}$. 
Then 
\begin{align*}
h=&\sum_{\substack{\alpha_0\alpha_1\alpha_2\subseteq \{0,\dots,10\}}}
\begin{pmatrix}
a_{\alpha_0}& a_{\alpha_1} &a_{\alpha_2}\\
b_{\alpha_0}& b_{\alpha_1} &b_{\alpha_2}
\end{pmatrix} 
+
\sum_{\substack{\beta_0\beta_1\beta_2\in {\rm STS}(9)}}
T(a_{\beta_0}b_{\beta_0},\,a_{\beta_1}b_{\beta_1},\,a_{\beta_2}b_{\beta_2},\,a_{\beta_0}b_{\beta_0})
 \\&\;+ T(a_{8}b_{8},\,a_{9}b_{9},\,a_{10}b_{10},\,a_{8}b_{8})\\
&\;
+
\sum_{\substack{a_ia_j\in S}}
T(a_{9}b_{9},\,a_{i}b_{i},\,a_{10}b_{10},\,a_{j}b_{j},\,a_{9}b_{9})
\end{align*}
is a  
$(2,3,22)$-halving.
Therefore, the trade structure of $(2,3,22)$-halving is the following:
\[\begin{aligned}165\, ({\rm minimal\; trades})+13\, ({\rm trades \;of\; volume\; 6})+4\,({\rm trades\; of\; volume\; 8})=770.\end{aligned}\] 
\section*{Acknowledgements} 
The authors thank Professor Denis Krotov of  Sobolev Institute of Mathematics for mentioning the paper by Sauskan and Tarannikov \cite{Packing}.

\end{document}